\documentclass[twoside,11pt]{article}
\usepackage{latexsym}
\usepackage{epsfig}
\usepackage{amsfonts}
\usepackage{amsthm,amsmath,amsfonts,amssymb}
\numberwithin{equation}{section}

\newcommand{\supp}{\operatorname{supp}}

\newcommand{\opm}{\operatorname{op}}

\begin{document}
\newtheorem{assumption}{Assumption}
\newtheorem{proposition}{Proposition}
\newtheorem{definition}{Definition}
\newtheorem{lemma}{Lemma}
\newtheorem{theorem}{Theorem}
\newtheorem{observation}{Observation}
\newtheorem{remark}{Remark}
\newtheorem{corollary}{Corollary}

\title{Explicit Green operators for quantum mechanical Hamiltonians.~I.~The hydrogen atom}

\author{Heinz-J{\"u}rgen Flad$^\dag$, Gohar Harutyunyan$^\ast$,
 Reinhold Schneider$^\dag$ and Bert-Wolfgang Schulze$^\ddag$ \\
\ \\
$^\dag${\small Institut f\"ur Mathematik, Technische Universit\"at Berlin,
Stra{\ss}e des 17. Juni 136, D-10623 Berlin}\\
$^\ast${\small Institut f\"ur Mathematik, Carl von Ossietzky Universit\"at
Oldenburg, D-26111 Oldenburg}\\
$^\ddag${\small Institut f\"ur Mathematik, Universit\"at Potsdam, Am Neuen Palais 10, D-14469 Potsdam}
}
\maketitle

\begin{abstract}
\noindent
We study a new approach to determine the asymptotic behaviour of quantum many-particle systems
near coalescence points of particles which interact via singular Coulomb potentials.
This problem is of fundamental interest in electronic structure theory in order to establish accurate
and efficient models for numerical simulations. Within our approach, coalescence points of particles
are treated as embedded geometric singularities in the configuration space of electrons.
Based on a general singular pseudo-differential calculus, we provide a recursive scheme
for the calculation of the parametrix and corresponding Green operator of a nonrelativistic
Hamiltonian. In our singular calculus, the Green operator encodes all the asymptotic information of the eigenfunctions.
Explicit calculations and an asymptotic representation for the Green operator of the hydrogen atom
and isoelectronic ions are presented.
\end{abstract}

\section{Introduction}
The asymptotic behaviour of eigenfunctions of nonrelativistic quantum mechanical Hamiltonians near singular points
of the potential is of great physical interest and paramount importance for the numerical solution of these eigenvalue problems.
This is especially true in the field of electronic structure calculations where electrons and
nuclei interact via singular Coulomb potentials with each other. We present a new approach to extract asymptotic
information by an explicit recursive calculation of the parametrix of the Hamiltonian and corresponding
Green operator. Such an approach, however, requires the full machinery of pseudo-differential
calculus on manifolds with singularities. The Coulomb potentials in the Hamiltonian induce a hierarchy of
embedded singularities into the configuration space of the electrons, i.e., $\mathbb{R}^{3N}$ for a system
of $N$ electrons within the Born-Oppenheimer approximation. This means that we have to consider
$\mathbb{R}^{3N}$ as a stratified space with conical, edge and higher order corner singularities.
The underlying general calculus of pseudo-differential operator algebras and corresponding function spaces has been
developed by one of the authors in great generality, cf.~the monographs \cite{ES97,Schulze98}.
It has been already applied to various types of boundary value problems with applications in material
science, i.e., crack theory, cf.~the joint monograph \cite{HS08} by two of the authors.
Applications to quantum theory are natural from a conceptual point of view, however,
they face considerable technical difficulties with regard to the tremendous complexity
of this calculus. Especially for higher order corner singularities it is still under development.
It is therefore the purpose of the present work to provide an explicit calculation
along the line of the general approach, which illustrates the basic idea and its potentiality.

The paper is organized as follows: first we briefly summarize previous work on the electronic
Schr\"odinger equation, focusing on the asymptotic behaviour near singularities of the potential.
In the subsequent section, we present a general outline of the basic ideas of the calculus thereby
avoiding most of the technicalities. In order to make the paper reasonably self-contained, we present
basic elements of the calculus in Appendix \ref{AppendixB}.
The main part of the paper is devoted to explicit
calculations for the hydrogen atom and isoelectronic ions, where we discuss the recursive construction
of the parametrix which enables us to extract the asymptotic information to any order.

\subsection{Previous work on the asymptotics of many-electron systems}
\label{summary}
The rigorous mathematical analysis of the asymptotic behaviour of eigenfunctions of nonrelativistic
Hamiltonians near singularities of the potential was initiated by Kato \cite{Kato57}.
Presently a considerably more detailed picture is available by the work
of M.~and T.~Hoffmann-Ostenhof and coworkers \cite{HOS81,HO292,HO2S94,HO2S01,FHO2S05,FHO2S09}.
Their work provides some leading order terms of asymptotic expansions and presents a source
of inspiration for numerical methods \cite{KMTV06}, which try to incorporate such terms explicitly into the
basis set. For certain purposes, however, like nonlinear approximation theory \cite{FHS07}, a more stringent
control of the singular behaviour is required. In this particular case, these results only provide
some guideline to conjecture the more general asymptotic behaviour.

There are only a few explicitly solvable electronic Coulomb systems mentioned in the literature.
The best known example is the hydrogen atom dicussed below. For systems with more than one electron, it is
only for the helium atom where an explicit solution has been conjectured by Fock \cite{Fock54}.
Despite important rigorous results in favour of Fock's expansion by Morgan \cite{Morg86}, a complete proof
is still missing. However, separating the leading order terms of the Fock expansion
enabled improved regularity results for the remaining part of general many-electron wavefunctions, cf.~\cite{FHO2S05}.
Unfortunately this approach is presently limited to leading order terms because spin-degrees of freedom
are not taken into account.

The conical singularity that occurs when an electron approaches a nucleus, which is considered
to be fixed within the Born-Oppenheimer approximation, is the simplest case in the hierarchy of singularities
discussed before. It is well-known that the eigenfunctions of the hydrogen atom possess a simple
Taylor asymptotic with respect to the electron-nuclear distance. Such an asymptotic behaviour has been
recently established for the nonlinear Hartree-Fock model \cite{FSS08} using tools from a singular pseudo-differential
operator calculus. This result provides, e.g., a solid ground for the application of nonlinear approximation theory
to singular particle models \cite{FHS06}.

\subsection{Fuchs type differential operators and their parametrices}
Due to space restrictions, we cannot provide a detailed account of the operator calculus on manifolds
with conical singularities. For the convenience of the reader we have summarized
some basic features of the calculus in Appendix \ref{AppendixB}. It is the purpose of this section
to give a brief outline of the main ideas underlying the present work.

The cone algebra over an infinite stretched cone $X^\wedge = \mathbb{R}_+ \times X$ for a closed compact
$C^\infty$ manifold X may be motivated by the task to establish a (pseudo-differential) operator calculus
that contains all Fuchs type cone degenerate differential operators of the form
\begin{equation}
 A = r^{-\mu} \sum\limits_{j=0}^\mu a_j(r) \biggl( -r \frac{\partial }{\partial  r} \biggr)^j 
\label{Amu}
\end{equation}
for $a_j \in C^\infty \bigl( \overline{\mathbb{R}}_+, \mbox{Diff}^{\mu-j}(X) \bigr)$, together with the parametrices
of elliptic elements. Here $\mbox{Diff}^\nu(X)$ denotes the space of differential operators on $X$ of
order $\nu$ with smooth coefficients in local coordinates. In such a context it is natural to admit
$\mu \in \mathbb{N}$ rather than $\mu=2$ which is the case for our applications,
since compositions and parametrices give rise to other orders.
Under some natural growth conditions on the coefficients $a_j(r)$ for $r \rightarrow \infty$ the
operator $A$ induces continuous operators between weighted Sobolev spaces
\[
 A: {\cal K}^{s,\gamma}(X^\wedge) \longrightarrow {\cal K}^{s-2,\gamma-2}(X^\wedge) ,
\]
for all $s,\gamma \in \mathbb{R}$. A brief discussion of these spaces can be found in Appendix \ref{Mellin}.
For a detailed account, we refer to the monographs \cite{ES97,Schulze98}.

Our approach to characterise asymptotics of eigenfunctions of cone degenerate operators (\ref{Amu})
is based on a singular calculus of reasonable generality
which means that it can be applied to a large class of problems where the present one is only a specific example.
Hereby it is a basic task to understand the structure of parametrices.
In general, asymptotics, together with weight information and smoothness of solutions,
may be expected from a parametrix $P$ of the respective elliptic operator $A$.
The left product $PA$ is equal to the identity plus a smoothing operator $G$, where the latter contains all the
information provided that $P$ is known in detail close to the singularity.
Within the present work, we restrict to the case of conical singularities with smooth base manifolds.
More general edge and corner singularities will be discussed in a subsequent publication.

In order to deal with parametrices it is necessary to consider a
wider class of pseudo-differential operators. The corresponding
symbols, e.g., for Fuchs type differential operators (\ref{Amu}),
are (up to the factor $r^{-\mu}$) locally of the form
\[
 p_A(r,x,\rho,\xi) \, = \, \tilde{p}(r,x,r\rho,\xi) ,
\]
where $\tilde{p}(r,x,\rho, \xi)$ is a classical pseudo-differential symbol in the covariables $(\rho, \xi)$,
smoothly depending on $(r,x)$ up to $r=0$. The Fuchs type symbol $p_A(r,x,\rho,\xi)$ is degenerate
for $r \rightarrow 0$ with respect to the covariable $\rho$.
Globally along the base $X$ of a cone one considers operator functions
$p(r,\rho) \, = \, \tilde{p}(r, r\rho)$
where $\tilde{p}(r, \rho)$ is a parameter-dependent family of classical pseudo-differential operators on $X$ with parameter
$\rho \in {\mathbb R}$, smoothly depending on $r \in \overline{\mathbb{R}}_+$, i.e.,
$\tilde{p}(r, \rho) \in C^{\infty} (\overline{\mathbb{R}}_+, L^{\mu}_{cl}(X; \mathbb{R}_{{\rho}}))$.

Elliptic Fuchs type differential operators (\ref{Amu}) can be alternatively expressed as Mellin pseudo-differential
operators along ${\mathbb R}_+$ with holomorphic Mellin symbols $h(r,z)$ taking values in (classical pseudo-)
differential operators on $X$. This admits the construction of a parametrix with symbol $h^{(-1)}(r,z)$ representing
the Mellin-Leibniz inverse of $h(r,z)$ with meromorphic Mellin symbol which has poles in the complex $z$-plane.
The poles and their multiplicities just determine the asymptotics of eigenfunctions near $r = 0$.

\section{The hydrogen atom revisited}
In nonrelativistic quantum theory, the hydrogen atom represents the most prominent textbook example
for a quantum mechanical Hamiltonian where all eigenfunctions belonging to discrete eigenvalues
are explicitly known. The stationary Schr\"odinger equation for the hydrogen atom and isoelectronic ions
has the form
\begin{equation}
 \bigl( H -E \bigr) u =0 ,
\label{Schroedinger}
\end{equation}
with Hamiltonian operator
\begin{equation}
 H = -\frac{1}{2} \Delta -\frac{Z}{|\tilde{x}|} 
\label{Hamiltonian}
\end{equation}
in atomic units, where $\Delta$ is the Laplacian in $\mathbb{R}^3 \ni \tilde{x}$ and $Z$
the charge of the nucleus. In polar coordinates, i.e.,
$(S^2)^{\wedge} := \mathbb{R}_+ \times S^2 \cong \mathbb{R}^3_{\tilde{x}} \setminus \{0\}$,
equation (\ref{Schroedinger}) becomes
\begin{equation}
 A u = \frac{1}{r^2} \biggl[ \frac{1}{2} \biggl(-r \frac{\partial}{\partial r} \biggr)^2 - \frac{1}{2}
 \biggl(-r \frac{\partial}{\partial r} \biggr) + \frac{1}{2} \Delta_{S^2} +rZ + r^2 E \biggr] u =0 .
\label{Aeq}
\end{equation}
It can be seen that the electron-nuclear Coulomb potential can be incorporated into the Fuchs type differential
operator acting on $(S^2)^{\wedge}$ with embedded conical singularity at the origin.

Let $\omega,\tilde{\omega}, \hat{\omega}$ be cut-off functions such
that $\omega\tilde{\omega}=\omega,\quad \omega \hat{\omega}=\hat{\omega}$.
Throughout this paper a cut-off function is any real-valued $\omega
\in C^\infty_0(\overline{\mathbb{R}}_+)$ such that $\omega(r) \equiv
1$ close to $r=0$.
It is easy to show that $A$ can be decomposed into two parts in the
following way
\[
A=\omega A\tilde{\omega}+(1-\omega)A(1-\hat{\omega}).
\]
Alternatively, the first term can be expressed as Mellin
pseudo-differential operator with holomorphic operator valued symbol 
\[
 h(r,w) = \tfrac{1}{2} \bigl( w^2 -w + \Delta_{S^2} \bigr) +rZ +r^2 E ,
\]
and $A$ in the second term can be
pushed forward to the differential operator $E-H$, i.e.,
\[
 A  = \omega r^{-2} \opm_M^{\gamma-1}(h) \tilde{\omega}+(1-\omega)(E-H)(1-\hat{\omega}).
\]
In the sequel we investigate $A$ as an operator in the cone algebra,
corresponding to the stretched cone $(S^2)^\wedge$. (We refer, for
example, to \cite[ Chapter 8]{ES97} for details  on such operators and
corresponding symbols; a brief outline is given in
Appendix \ref{Mellin}.)

As a crucial point first we consider the  bijectivity of the
principal conormal symbol of the Mellin operator
\[
 \sigma_M(A)(w) = \tfrac{1}{2} \left( w^2 -w + \Delta_{S^2} \right),
\]
as a parameter dependent second order differential operator between
standard Sobolev spaces on $S^2$.

\begin{lemma}
The principal conormal symbol defines isomorphisms
\begin{equation}\label{fr}
\sigma_M(A)(w): H^s(S^2) \rightarrow H^{s-2}(S^2)
\end{equation}
for any $s\in\mathbb{R}$ provided that $w\notin \mathbb{Z}$.
\label{lemma1}
\end{lemma}
\begin{proof}
From the ellipticity of $\sigma_M(A)$ follows the Fredholm property
of (\ref{fr}) for any $s\in\mathbb{R}$.  Therefore it is sufficient
to consider the (of $s$ independent)  kernels of (\ref{fr}) and its
adjoint. Let us take $s=2.$ The eigenfunctions of the operator
$\Delta_{S^2}: H^2(S^2) \rightarrow L^2(S^2)$ are the spherical
harmonics $\{Y_{lm} : l\in \mathbb{N}_0, -l \leq m \leq l \}$
corresponding to the $2l+1$ times degenerate eigenvalues $-l(l+1)$.
 These functions form a complete orthogonal basis in $L^2(S^2)$ and belong to $H^\infty(S^2)$.
A simple calculation shows that the kernel of
$\sigma_M(A)$ has no nontrivial element for $w \notin \mathbb{Z}$. The adjoint
operator $\sigma_M(A)^\ast: L^2(S^2) \rightarrow H^{-2}(S^2)$ satisfies
$\langle u , \sigma_M(A)^\ast v \rangle = \langle \sigma_M(A) u , v \rangle$
for all $u \in H^2(S^2)$.
If $v$ belongs to the kernel of $\sigma_M(A)^\ast$, we get $\langle \sigma_M(A)u , v \rangle =0$ for any $u \in H^2(S^2)$; in particular,
\[
 \langle \sigma_M(A) Y_{lm} , v \rangle =
 \tfrac{1}{2} \left( w^2 -w -l(l+1) \right) \langle Y_{lm} , v \rangle=0 .
\]
It follows that if $w\notin \mathbb{Z}$ we have $\langle Y_{lm} , v \rangle=0$ for all $l,m$ and therefore $v=0$.
\end{proof}

Let us now check the ellipticity conditions for $A$ (as  an element
in $C^2((S^2)^\wedge,\boldsymbol{g}),
\boldsymbol{g}=(\gamma,\gamma-2,\Theta)$ with $\Theta=(-\infty,0]$)
in the sense of \cite[Definition 6, Section 8.2.5]{ES97}. The first
condition is obviously satisfied. The second condition follows for
any $\gamma\notin \mathbb{Z}+\frac{1}{2}$ by Lemma 1.
It only remains to check the exit behaviour of $A$.  The
corresponding exit symbols
\[
 \sigma_e({A}) = -\tfrac{1}{2} \left( \xi_1^2 + \xi_2^2 + \xi_3^2 \right) +E \ \ \mbox{and} \
 \sigma_{\psi,e}({A}) = -\tfrac{1}{2} \left( \xi_1^2 + \xi_2^2 + \xi_3^2 \right) ,
\]
have constant coefficients and satisfy for $E< 0$ (which is the case in our application) the required exit conditions, i.e.,
\[
 \sigma_e({A}) \neq 0 \ \ \mbox{for all} \ \xi \in \mathbb{R}^3 \ \ \mbox{and} \
 \sigma_{\psi,e}({A}) \neq 0 \ \ \mbox{for all} \ \xi \in \mathbb{R}^3 \setminus \{0\} .
\]
From the ellipticity of $A$ it follows the existence of a parametrix
and corresponding Green operator. The explicit construction of the
parametrix requires the inverse of the principal conormal symbol
$\sigma_M(A)$ as a meromorphic operator function with poles at the
non-bijectivity points $w_0 \in \mathbb{Z}$.

\begin{lemma}
\label{lemma2} The inverse of the principal conormal symbol
$\sigma_M(A)^{-1}(w)$ is a meromorphic operator function with simple
poles at $w_0 \in \mathbb{Z}$; more precisely,
\[
 \Xi \sigma_M(A)^{-1}(w) = -2(w+l)^{-1} \sum_{m=-l}^l \langle \cdot, \tilde{Y}_{lm} \rangle
 Y_{lm} \ \ \mbox{for} \ w_0 \leq 0 \ \mbox{and} \ l=-w_0
\]
\[
 \Xi \sigma_M(A)^{-1}(w) = 2(w-l-1)^{-1} \sum_{m=-l}^l \langle \cdot, \tilde{Y}_{lm} \rangle
 Y_{lm} \ \ \mbox{for} \ w_0 > 0 \ \mbox{and} \ l=w_0-1
\]
where $\tilde{Y}_{lm} := \tfrac{1}{2l+1} Y_{lm}$ and $\Xi$ denotes the principal part of the Laurent expansion
of $\sigma_M(A)^{-1}(w)$ in a neighbourhood of the pole $w_0$.
\end{lemma}
\begin{proof}
The lemma is an immediate consequence of Theorem 7.1 in \cite{GS71}. In order to apply the theorem it
is only necessary to calculate at $w=w_0$ the maximum of the null multiplicity of
\[
 \left( w^2 -w + \Delta_{S^2} \right) f_{lm}(w,\theta,\phi) ,
\]
for arbitrary $w$ holomorphic functions with $f_{lm}(w_0,\theta,\phi) = Y_{lm}(\theta,\phi)$.
Suppose the first derivative with respect to $w$ vanishes at $w=w_0$, i.e.,
\[
 \left( 2w_0 -1 \right) f_{lm}(w_0,\theta,\phi) + \left( w_0^2 -w_0 + \Delta_{S^2} \right) f'_{lm}(w_0,\theta,\phi) =0 ,
\]
which can be written as
\[
 \left( w_0^2 -w_0 + \Delta_{S^2} \right) f'_{lm}(w_0,\theta,\phi) = - \left( 2w_0 -1 \right) Y_{lm}(\theta,\phi) .
\]
This equation has no solution for $f'_{lm}(w_0,\theta,\phi)$ since the right side belongs to the kernel of the
self-adjoint operator on the left side. Therefore all null multiplicities are of first order.
\end{proof}

According to the spectral theorem for operators with purely
pointwise spectrum (cf. \cite{Triebel}), we can represent the
inverse of the principal conormal symbol by the absolutely
convergent sum
\begin{equation}
 \sigma_M(A)^{-1}(w) = 2 \sum_{l=0}^\infty \frac{P_l}{w^2-w -l(l+1)} \ \ \mbox{with} \
 P_l := \sum_{m=-l}^l \langle \cdot , Y_{lm} \rangle Y_{lm}.
\label{sMAm1}
\end{equation}
An alternative simple proof of Lemma \ref{lemma2} can be given using the explicit representation (\ref{sMAm1}).

\subsection{Construction of the parametrix and Green operator}
\label{parametrix} The operator $A$ belongs to the cone algebra on
$(S^2)^\wedge$ and defines bounded operators between weighted
Sobolev spaces
\[
 A: {\cal K}^{s,\gamma}((S^2)^\wedge) \longrightarrow {\cal K}^{s-2,\gamma-2}((S^2)^\wedge) .
\]
In the following, we require the corresponding weighted Sobolev
spaces with asymptotics ${\cal K}^{s,\gamma}_Q ((S^2)^\wedge)$.
These spaces can be considered as  direct sums
\begin{equation}
 {\cal K}^{s,\gamma}_Q ((S^2)^\wedge) = {\cal E}^\gamma_Q + {\cal K}^{s,\gamma}_\Theta ((S^2)^\wedge) ,
\label{E+K}
\end{equation}
of the flattened weighted Sobolev spaces
\[
 {\cal K}^{s,\gamma}_\Theta ((S^2)^\wedge) := \bigcap_{\epsilon > 0} {\cal K}^{s,\gamma - \vartheta - \epsilon}
 ((S^2)^\wedge) ,
\]
with $\Theta =(\vartheta,0]$, $-\infty \leq \vartheta < 0$, and
asymptotic spaces
\[
 {\cal E}^\gamma_Q := \biggl\{ \omega \sum_j \sum_{k=0}^{m_j} c_{jk}(x) r^{-q_j} \ln^k r \biggr\} .
\]
The asymptotic space ${\cal E}^\gamma_Q$ is characterized by a sequence $q_j \in \mathbb{C}$
which is taken from a strip of the complex plane, i.e.,
\[
 q_j \in \left\{ z: \frac{3}{2}-\gamma + \vartheta < \Re z < \frac{3}{2}-\gamma \right\} ,
\]
where the width and location of this strip are determined by its {\em weight data} $(\gamma,\Theta)$
with $\Theta =(\vartheta,0]$ and $-\infty \leq \vartheta < 0$. Each substrip of finite width
contains only a finite number of $q_j$. Furthermore, the coefficients
$c_{jk}$ belong to finite dimensional subspaces $L_j \subset C^\infty(S^2)$.
The asymptotics of ${\cal E}^\gamma_Q$ is therefore completely
characterized by the {\em asymptotic type} $Q := \{(p_j,m_j,L_j)\}_{j \in \mathbb{Z}_+}$.
In the following, we employ the asymptotic subspaces
\begin{equation}
 {\cal S}^\gamma_Q ((S^2)^\wedge) := \left\{ u \in {\cal K}^{\infty,\gamma}_Q ((S^2)^\wedge) :
 (1- \omega) u \in {\cal S}(\mathbb{R},C^\infty(S^2))|_{\mathbb{R}_+ \times S^2} \right\} ,
\label{Sg}
\end{equation}
with Schwartz class type of behaviour for exit $r \rightarrow
\infty$. The spaces ${\cal K}^{s,\gamma}_Q((S^2)^\wedge)$ and ${\cal
S}^\gamma_Q ((S^2)^\wedge)$ are Fr\'echet spaces equipped with
natural semi-norms according to the decomposition (\ref{E+K}); we
refer to \cite{ES97, Schulze98} for further datails.

As an elliptic element of the operator class
$C^2((S^2)^\wedge,\boldsymbol{g})$, $\boldsymbol{g} =(\gamma,
\gamma-2, \Theta)$ with $\gamma\notin \mathbb{Z}+\frac{1}{2}$ and
$\Theta=(-\infty,0]$, $A$ has a parametrix $P$ in the cone algebra
which belongs to $C^{-2}((S^2)^\wedge,\boldsymbol{g})$,
$\boldsymbol{g} =(\gamma-2, \gamma, \Theta)$. It can be written in
the general form
\[
 P = \omega' r^2 \opm_M^{\gamma-3} \bigl( h^{(-1)}(r,w) \bigr) \tilde{\omega}' + \bigl( 1- \omega' \bigr)
 \tilde{P} \bigl( 1- \hat{\omega}' \bigr) ,
\]
where $\omega'$, $\tilde{\omega}'$, $\hat{\omega}'$ are cut-off
functions satisfying $\omega' \tilde{\omega}' = \omega'$, $\omega'
\hat{\omega}' = \hat{\omega}'$ and $\tilde{P}$ represents a standard
pseudo-differential operator of order $-2$ on $\mathbb{R}^3$. By
definition, the parametrix satisfies the equation
\[
 P A = 1 +G \ \ \mbox{with} \ G \in C_G((S^2)^\wedge, \boldsymbol{g}_l), \ \boldsymbol{g}_l=(\gamma,\gamma,\Theta) ,
\]
where the Green operator $G$ maps ${\cal
K}^{s,\gamma}((S^2)^\wedge)$ into ${\cal S}^\gamma_Q$ for some
discrete asymptotics $Q$. Application of the parametrix to
(\ref{Aeq}) yields $u =-G u$, i.e., the asymptotic behaviour of the
eigenfunctions belonging to the eigenvalue $E$ are completely
characterized by the corresponding Green operator.

In order to proceed with our discussion, we need an educated guess
for the parameter $\gamma$. Taking $\gamma$ to large, it might
happen that a solution of $A u =0$ does not exist in ${\cal
K}^{s,\gamma}$ and a too small value for $\gamma$ may introduce
additional solutions which are too singular to be meaningful from a
physical viewpoint. In general it should be not too difficult to
identify a priori an appropriate interval for $\gamma$ using
standard regularity results. In particular, for the hydrogen atom explicitly known  eigenfunctions belong to ${\cal K}^{s,\gamma}$ with $\gamma
< \frac{3}{2}$. Therefore we assume $\frac{1}{2} < \gamma <
\frac{3}{2}$ in the following.

After we have established the notion of a parametrix and Green operator in the pseudo-differential algebra,
it is possible to state our main theorem.

\begin{theorem}
\label{theorem1} The Hamiltonian of the hydrogen atom
{\em(}$Z=1${\em )} and for isoelectronic ions {\em
(\ref{Hamiltonian})} belongs to the cone algebra, i.e., it belongs
to the operator class $C^2((S^2)^\wedge,\boldsymbol{g})$,
$\boldsymbol{g} =(\gamma, \gamma-2, \Theta),$ and is elliptic in the
sense of {\em \cite[Definition 6, Section 8.2.5]{ES97}} for $\gamma
\not\in \mathbb{Z} +\frac{1}{2}$. There exist a corresponding
parametrix $P \in C^{-2}((S^2)^\wedge,\boldsymbol{g}^{-1})$,
$\boldsymbol{g}^{-1} =(\gamma-2, \gamma, \Theta)$ and Green operator
$G \in C_G((S^2)^\wedge, \boldsymbol{g}_l)$,
$\boldsymbol{g}_l=(\gamma,\gamma,\Theta)$. The Green operator for
$\frac{1}{2} < \gamma < \frac{3}{2}$ has a leading order asymptotic
expansion of the form \setlength{\jot}{5mm}
\begin{eqnarray}
\nonumber
 G(u) & = & \omega' \left(1-Zr+(Z^2-E) \tfrac{1}{3}r^2 + \cdots \right) P_0 {\cal Q}_0(u) \\ \label{Ghydrogen}
 &  & +\omega' \left( Zr -\tfrac{1}{2} (Zr)^2 +\tfrac{1}{10}(Z^2-2E) Zr^3 + \cdots \right)
 P_1 {\cal Q}_1(u) + \ \cdots
\end{eqnarray}
\setlength{\jot}{3mm}
with coefficients
\begin{eqnarray*}
 {\cal Q}_0(u) & = & -2 \left[ \bigl( M ( \omega'' -1) \opm_M^{\gamma-1} \bigl( h_0 \bigr) (\tilde{\omega}u) \bigr)(0)
 +\bigl( M \omega'' (rZ+r^2E) u \bigr)(0) \right] \\
 {\cal Q}_1(u) & = & -\tfrac{2}{3} Z^{-1} \left[ \bigl( M ( \omega'' -1) \opm_M^{\gamma-1}
 \bigl( h_0 +Zr \bigr) (\tilde{\omega}u) \bigr)(-1) +
 \bigl( M \omega'' r^2E u \bigr)(-1) \right] \\
 & & - \tfrac{2}{3} \left[ \bigl( M ( \omega'' -1) \opm_M^{\gamma-1} \bigl( h_0 \bigr) (\tilde{\omega}u) \bigr)(0) +
 \bigl( M \omega'' (rZ+r^2E) u \bigr)(0) \right]
\end{eqnarray*}
given by values of Mellin transformations $M$ {\em(}cf.~Appendix \ref{Mellin}{\em)} at non-bijectivity points of the principal
conormal symbol $\sigma_M(A)$,
where $u \in {\cal K}^{s,\gamma}((S^2)^\wedge)$, $\omega'' := \tilde{\omega}' \omega$, and
$P_l$, $l=0,1,2,\ldots$,
denotes projection operators on subspaces which belong to eigenvalues
$-l(l+1)$ of the Laplace-Beltrami operator on $S^2$.
\end{theorem}

\begin{remark}
\label{remark1}
 Using the explicit expression for eigenvalues of the
hydrogen atom and isoelectronic ions, i.e., $E_n= -\frac{1}{2}
\frac{Z^2}{n^2}$, $n \in \mathbb{N}$, we obtain
\setlength{\jot}{5mm}
\begin{eqnarray*}
 G_n(u) & = & \omega' \left(1-Zr+(1-\tfrac{1}{2n^2}) \tfrac{1}{3}(Zr)^2 + \cdots \right) P_0 {\cal Q}_0(u) \\
 &  & +\omega' \left( Zr -\tfrac{1}{2} (Zr)^2 +\tfrac{1}{10}(1-\tfrac{1}{n^2}) (Zr)^3 + \cdots \right)
 P_1 {\cal Q}_1(u) + \ \cdots
\end{eqnarray*}
\setlength{\jot}{3mm} which can be easily verified for asymptotic
expansions of the analytic expressions of the eigenfunctions for
ground and excited states with angular momentum quantum number $l=0$
and $1$, cf.~the monograph {\em \cite{BS57}}.
\end{remark}

In order to calculate the asymptotics of the Green operator it is sufficient to focus on a neighbourhood
of the embedded conical singularity at $r=0$. Using the general commutation relation
\begin{equation}\label{gohar}
 r^n \opm_M^{\beta} \bigl( g(r,w) \bigr) r^{-n} = \opm_M^{\beta+n} \bigl( T^n g(r,w) \bigr)
\end{equation}
with shift operator $T^n$ acting on the Mellin symbol by $T^n g(r,w) = g(r,w+n)$, we get
\begin{eqnarray}
\nonumber
 P A &=& \omega' r^2 \opm_M^{\gamma-3} \bigl( h^{(-1)}(r,w) \bigr) \tilde{\omega}'
 \omega r^{-2} \opm_M^{\gamma-1} \bigl( h(r,w) \bigr) \tilde{\omega} + \cdots \\ \label{PA}
 &=& \omega' \opm_M^{\gamma-1} \bigl( T^2 h^{(-1)}(r,w) \bigr) \tilde{\omega}'
 \omega \opm_M^{\gamma-1} \bigl( h(r,w) \bigr) \tilde{\omega} + \cdots
\end{eqnarray}
where the omitted terms do not contribute to the asymptotics of the Green operator.
For the following calculations it is convenient to change our notation, we define $h_0(w) := \sigma_M(A)(w)$,
$h_1(w) := Z$ and $h_2(w) := E$. Herewith the holomorphic symbol becomes
\[
 h(r,w) = h_0(w) +r h_1(w) +r^2 h_2(w) .
\]
Similarly, we perform a Taylor series expansion of the Mellin symbol of the parametrix
\[
 h^{(-1)}(r,w) = \sum_{i=0}^N r^i h^{(-1)}_i(w) + R_N(r,w) .
\]
Inserting both expansions into (\ref{PA}) yields
\begin{eqnarray*}
 P A &=& \omega' \opm_M^{\gamma-1} \biggl( \sum_{i=0}^N r^i T^2 h^{(-1)}_i(w) + T^2 R_N(r,w) \biggr) \tilde{\omega}'
 \omega \opm_M^{\gamma-1} \bigl( h_0(w) +r h_1(w) +r^2 h_2(w) \bigr) \tilde{\omega} + \cdots \\
 &=& \omega' \sum_{i=0}^N r^i \opm_M^{\gamma-1} \bigl( T^2 h^{(-1)}_i(w) \bigr) \tilde{\omega}'
 \omega \opm_M^{\gamma-1} \bigl( h_0(w) +r h_1(w) +r^2 h_2(w) \bigr) \tilde{\omega} + \cdots \\
 &=& \omega' \sum_{i=0}^N r^i \opm_M^{\gamma-1} \bigl( T^2 h^{(-1)}_i(w) \bigr) \tilde{\omega}'
 \omega \opm_M^{\gamma-1} \bigl( h_0(w) \bigr) \tilde{\omega} \\
 & & +\omega' \sum_{i=0}^N r^{i+1} \opm_M^{\gamma-2} \bigl( T^1 h^{(-1)}_i(w) \bigr) \tilde{\omega}'
 \omega \opm_M^{\gamma-1} \bigl( h_1(w) \bigr) \tilde{\omega} \\
 & & +\omega' \sum_{i=0}^N r^{i+2} \opm_M^{\gamma-3} \bigl( h^{(-1)}_i(w) \bigr) \tilde{\omega}'
 \omega \opm_M^{\gamma-1} \bigl( h_2(w) \bigr) \tilde{\omega} + \cdots
\end{eqnarray*}
The last two lines contain products of Mellin operators with different weights. Adding and
subtracting terms with Mellin operators of the same weight yields
\begin{eqnarray*}
 PA &=& \omega' \sum_{i=0}^N r^i \opm_M^{\gamma-1} \bigl( T^2 h^{(-1)}_i(w) \bigr) \tilde{\omega}'
 \omega \opm_M^{\gamma-1} \bigl( h_0(w) \bigr) \tilde{\omega} \\
 & & +\omega' \sum_{i=0}^N r^{i+1} \opm_M^{\gamma-1} \bigl( T^1 h^{(-1)}_i(w) \bigr) \tilde{\omega}'
 \omega \opm_M^{\gamma-1} \bigl( h_1(w) \bigr) \tilde{\omega} \\
 & & +\omega' \sum_{i=0}^N r^{i+2} \opm_M^{\gamma-1} \bigl( h^{(-1)}_i(w) \bigr) \tilde{\omega}'
 \omega \opm_M^{\gamma-1} \bigl( h_2(w) \bigr) \tilde{\omega} \\
 & & \left. \!\! \begin{split}
 +& \sum_{i=0}^N r^{i+1} \left[ \omega' \opm_M^{\gamma-2} \bigl( T^1 h^{(-1)}_i(w) \bigr) \tilde{\omega}'
 - \omega' \opm_M^{\gamma-1} \bigl( T^1 h^{(-1)}_i(w) \bigr) \tilde{\omega}' \right]
 \omega \opm_M^{\gamma-1} \bigl( h_1(w) \bigr) \tilde{\omega} \\
 +& \sum_{i=0}^N r^{i+2} \left[ \omega' \opm_M^{\gamma-3} \bigl( h^{(-1)}_i(w) \bigr) \tilde{\omega}'
 - \omega' \opm_M^{\gamma-1} \bigl( h^{(-1)}_i(w) \bigr) \tilde{\omega}' \right]
 \omega \opm_M^{\gamma-1} \bigl( h_2(w) \bigr) \tilde{\omega}
 \end{split} \right\} =: G_I \\
 & & + \cdots ,
\end{eqnarray*}
where the operators in the last two lines will be identified as a part $G_I$ of the Green operator $G$.
In the next step we identify the remaining part of the Green operator $G$
\begin{eqnarray*}
 P A &=& \omega' \sum_{i=0}^N r^i \opm_M^{\gamma-1} \bigl( T^2 h^{(-1)}_i(w) \bigr)
 \opm_M^{\gamma-1} \bigl( h_0(w) \bigr) \tilde{\omega} \\
 & & +\omega' \sum_{i=0}^N r^{i+1} \opm_M^{\gamma-1} \bigl( T^1 h^{(-1)}_i(w) \bigr)
 \opm_M^{\gamma-1} \bigl( h_1(w) \bigr) \tilde{\omega} \\
 & & +\omega' \sum_{i=0}^N r^{i+2} \opm_M^{\gamma-1} \bigl( h^{(-1)}_i(w) \bigr)
 \opm_M^{\gamma-1} \bigl( h_2(w) \bigr) \tilde{\omega}\\
 && +G_I \\
 & & \left. \!\! \begin{split}
 +& \omega' \sum_{i=0}^N r^i \opm_M^{\gamma-1} \bigl( T^2 h^{(-1)}_i(w) \bigr) \bigl( \tilde{\omega}'
 \omega-1 \bigr) \opm_M^{\gamma-1} \bigl( h_0(w) \bigr) \tilde{\omega} \\
 +& \omega' \sum_{i=0}^N r^{i+1} \opm_M^{\gamma-1} \bigl( T^1 h^{(-1)}_i(w) \bigr) \bigl( \tilde{\omega}'
 \omega-1 \bigr) \opm_M^{\gamma-1} \bigl( h_1(w) \bigr) \tilde{\omega} \\
 +& \omega' \sum_{i=0}^N r^{i+2} \opm_M^{\gamma-1} \bigl( h^{(-1)}_i(w) \bigr) \bigl( \tilde{\omega}'
 \omega-1 \bigr) \opm_M^{\gamma-1} \bigl( h_2(w) \bigr) \tilde{\omega}
 \end{split} \right\} =: G_{II} \\
 & & + \cdots .
\end{eqnarray*}
The individual operators in the first three lines and the omitted terms are classified according to their
asymptotic behaviour for $r \rightarrow 0$. Therefore, the term corresponding to $r^0$ is equal to the
identity operator and all other terms corresponding to $r^i$, $i=1,2,\ldots$, must be zero.
This provides a recursive scheme to calculate the meromorphic operator valued symbols $h^{(-1)}_i(w)$ of the parametrix
explicitly. For $r^0$ we get
\[
 1= \opm_M^{\gamma-1} \bigl( T^2 h^{(-1)}_0(w) \bigr) \opm_M^{\gamma-1} \bigl( h_0(w) \bigr)
 = \opm_M^{\gamma-1} \bigl( T^2 h^{(-1)}_0(w) h_0(w) \bigr)
\]
and therefore
\begin{equation}
 \Big(T^2 h^{(-1)}_0(w)\Big) h_0(w) =1 \ \longrightarrow \ h^{(-1)}_0(w) = T^{-2} h_0^{-1}(w)=h_0^{-1}(w-2) .
\label{h0m1}
\end{equation}
The condition for $j=1$ is
\[
 0= \opm_M^{\gamma-1} \bigl( T^2 h^{(-1)}_1(w) \bigr) \opm_M^{\gamma-1} \bigl( h_0(w) \bigr) +
 \opm_M^{\gamma-1} \bigl( T^1 h^{(-1)}_0(w) \bigr) \opm_M^{\gamma-1} \bigl( h_1(w) \bigr) ,
\]
or
\[
 0= \Big(T^2 h^{(-1)}_1(w)\Big) h_0(w) + \Big(T^1 h^{(-1)}_0(w)\Big) h_1(w) ,
\]
which gives us
\begin{eqnarray}
\nonumber
 h^{(-1)}_1(w) & = & -\Big( T^{-1} h^{(-1)}_0(w)\Big)\Big( T^{-2} h_1(w)\Big)\Big( T^{-2} h_0^{-1}(w)\Big)\\ \label{h1m1}
 & = & -Z h^{-1}_0(w-3) h^{-1}_0(w-2) .
\end{eqnarray}
Similarly for $j=2$
\begin{eqnarray*}
 0 & = & \opm_M^{\gamma-1} \bigl( T^2 h^{(-1)}_2(w) \bigr) \opm_M^{\gamma-1} \bigl( h_0(w) \bigr) +
 \opm_M^{\gamma-1} \bigl( T^1 h^{(-1)}_1(w) \bigr) \opm_M^{\gamma-1} \bigl( h_1(w) \bigr) \\
 & & + \opm_M^{\gamma-1} \bigl( h^{(-1)}_0(w) \bigr) \opm_M^{\gamma-1} \bigl( h_2(w) \bigr) ,
\end{eqnarray*}
or
\[
 0= \Big(T^2 h^{(-1)}_2(w)\Big) h_0(w) +\Big(T^1 h^{(-1)}_1(w)\Big) h_1(w) + h^{(-1)}_0(w) h_2(w) ,
\]
which together with (\ref{h0m1}) and (\ref{h1m1}) yields
\begin{eqnarray}
\nonumber
 h^{(-1)}_2(w) & = & - \left( Z T^{-1} h^{(-1)}_1(w) + E T^{-2} h^{(-1)}_0(w) \right) T^{-2} h^{-1}_0(w) \\ \label{h2m1}
 & = & Z^2 h^{-1}_0(w-4) h^{-1}_0(w-3) h^{-1}_0(w-2) - E h_0^{-1}(w-4) h^{-1}_0(w-2) .
\end{eqnarray}
Obviously, the recursive scheme can be easily continued up to arbitrary order. For our present work
we consider only terms up to $j=2$ which is sufficient to demonstrate asymptotic properties of the
Green operator. Using the spectral decomposition (\ref{sMAm1}), we finally obtain
\begin{gather}
\label{hm1i0}
 h^{(-1)}_0(w) = 2 \sum_{l=0}^\infty \frac{P_l}{(w+l-2)(w-l-3)}, \\ \label{hm1i1}
 h^{(-1)}_1(w) = -4Z \sum_{l=0}^\infty \frac{P_l}{(w+l-3)(w+l-2)(w-l-3)(w-l-4)}, \\ \label{hm1i2}
 h^{(-1)}_2(w) = \sum_{l=0}^\infty \left( \frac{8Z^2}{(w+l-3)(w-l-4)} -4E \right)
 \frac{P_l}{(w+l-4)(w+l-2)(w-l-3)(w-l-5)}.
\end{gather}
Because of $P_l P_{l'}=0$ for $l \neq l'$, the poles of $h^{(-1)}_1(w)$ and $h^{(-1)}_2(w)$ remain simple.

Next, we consider $G_I$ and $G_{II}$-type Green operators in more detail.
In the following let us set $\Gamma_\beta := \{ w \in \mathbb{C} : \Re w = \beta \}$ for any $\beta \in \mathbb{R}$.
The $G_I$ operator contains operators of the general form
\begin{eqnarray*}
 & & \omega' \opm_M^{\tilde{\gamma}-1} \bigl( h(w) \bigr) \tilde{\omega}'
 - \omega' \opm_M^{\gamma-1} \bigl( h(w) \bigr) \tilde{\omega}' \\
 & = & \left[ \omega' \opm_M^{\tilde{\gamma}-1} \bigl( h(w) \bigr) r^{\gamma - \tilde{\gamma}} \tilde{\omega}'
 - \omega' r^{\gamma - \tilde{\gamma}} \opm_M^{\tilde{\gamma}-1} \bigl( T^{-\gamma + \tilde{\gamma}} h(w) \bigr)
 \tilde{\omega}' \right] r^{-\gamma + \tilde{\gamma}}.
\end{eqnarray*}
 According to \cite[Proposition
2.3.69]{Schulze98}, the term in square brackets belongs to $C_G((S^2)^\wedge,\boldsymbol{g})$ with $\boldsymbol{g}=(\tilde{\gamma},\tilde{\gamma},(-\infty,0]).$ By repeating  the calculations in the proof of the proposition mentioned before, 
in our case $\gamma>\tilde{\gamma},$ the whole expression can be transformed into
\begin{eqnarray}
\nonumber
 \lefteqn{\omega' \opm_M^{\tilde{\gamma}-1} \bigl( h(w) \bigr) v
 - \omega' r^{\gamma - \tilde{\gamma}} \opm_M^{\tilde{\gamma}-1} \bigl( T^{-\gamma + \tilde{\gamma}} h(w) \bigr)
 r^{-\gamma + \tilde{\gamma}} v} \\ \label{GI}
 & & \hspace{4cm} = \frac{1}{2\pi i} \omega' \int_{\Gamma^{-1}_{\frac{3}{2}-\gamma} \cup \Gamma_{\frac{3}{2}-\tilde{\gamma}}}
 r^{-w} h(w) \bigl( M_{\tilde{\gamma}-1} v \bigr)(w) dw ,
\end{eqnarray}
with $v:= \tilde{\omega}' u$ and $u\ \in {\cal
K}^{s,\gamma}((S^2)^\wedge)$. Since $\bigl( M_{\tilde{\gamma}-1}v
\bigr)(w)$ belongs to ${\cal A} \bigl( \frac{3}{2}-\gamma < \Re w <
\frac{3}{2}- \tilde{\gamma}, C^\infty(S^2) \bigr)$, we can apply
Cauchy's theorem to replace the integral contour by a closed smooth
curve $S_{\gamma,\tilde{\gamma}} \subset \{w \in \mathbb{C}:
\frac{3}{2}-\gamma < \Re w < \frac{3}{2}- \tilde{\gamma} \}$ which
encloses all poles in the strip and compute the integral using
Cauchy's formula.

The $G_{II}$-type Green operators can be computed in a similar manner. The operators are of general form
\[
 \omega' \opm_M^{\gamma-1} \bigl( g(w) \bigr) \bigl( \omega'' -1 \bigr)
 \opm_M^{\gamma-1} \bigl( h(w) \bigr) \tilde{\omega} ,
\]
with $\omega'' := \tilde{\omega}' \omega$ and $g(w) \in
M^\mu_Q(S^2)$, $h(w) \in M^\nu_R(S^2)$ (cf.~Appendix
\ref{structurepara}) for some $\mu, \nu \in \mathbb{R}$ and discrete
asymptotics $R,Q$. Because of $\omega' \tilde{\omega}' = \omega'$
and $\omega \tilde{\omega} = \omega$ we can assume w.l.o.g.~that
$\supp \omega' \cap \supp (1-\omega'') = \emptyset$. It follows from
\cite[Lemma 2.3.73]{Schulze98} that this operator belongs to
$C_G((S^2)^\wedge,\boldsymbol{g})$ with $\boldsymbol{g}
=(\gamma,\gamma,(-\infty,0])$. Following the proof in
\cite{Schulze98} we choose a cut-off function $\hat{\omega}$ with
$(1-\omega'') (1-\hat{\omega}) = (1-\omega'')$ and observe that
\[
 (1-\hat{\omega}) \opm_M^{\gamma-1} \bigl( h(w) \bigr) \tilde{\omega} :
 {\cal K}^{s,\gamma}\bigl( (S^2)^\wedge \bigr) \longrightarrow \omega'' {\cal K}^{s-\nu,\gamma+N}\bigl( (S^2)^\wedge \bigr)
 + (1- \omega'') {\cal H}^{s-\nu,\gamma}\bigl( (S^2)^\wedge \bigr)
\]
is continuous for every $N \in \mathbb{N}$. Furthermore
\[
 \omega' \opm_M^{\gamma-1} \bigl( g(w) \bigr) \bigl( \omega'' -1 \bigr):
 \omega'' {\cal K}^{s-\nu,\gamma+N}\bigl( (S^2)^\wedge \bigr)
 + (1- \omega'') {\cal H}^{s-\nu,\gamma}\bigl( (S^2)^\wedge \bigr) \longrightarrow
 {\cal S}^{\gamma}_{P_k}\bigl( (S^2)^\wedge \bigr)
\]
is continuous, where the asymptotic type $P_k$ has weight data $(\gamma,(-k,0])$ with
$k \in \mathbb{N}$ sufficiently large for $N$.
It follows from these considerations that it is the mapping
\[
 \omega' \opm_M^{\gamma-1} \bigl( g(w) \bigr) \bigl( \omega'' -1 \bigr):
 {\cal H}^{s-\nu,\gamma}\bigl( (S^2)^\wedge \bigr) \longrightarrow
 {\cal S}^{\gamma}_{P_k}\bigl( (S^2)^\wedge \bigr) ,
\]
which determines the asymptotics of the Green operator. For $u \in {\cal H}^{s-\nu,\gamma}\bigl( (S^2)^\wedge \bigr)$
and every $\tilde{\gamma} > \gamma$ it follows that $\bigl( M (\omega'' -1) u \bigr)(w)$ belongs to
${\cal A} \bigl( \frac{3}{2}-\tilde{\gamma} < \Re w < \frac{3}{2}- \gamma, C^\infty(S^2) \bigr)$.
Because of
${\cal H}^{s-\nu,\gamma}\bigl( (S^2)^\wedge \bigr) \subset {\cal H}^{s-\nu,\tilde{\gamma}}\bigl( (S^2)^\wedge \bigr)$,
we can rearrange the operator into
\[
 \omega' \opm_M^{\gamma-1} \bigl( g(w) \bigr) ( \omega'' -1) =
 \omega' \left[ \opm_M^{\gamma-1} \bigl( g(w) \bigr) -
 \opm_M^{\tilde{\gamma}-1} \bigl( g(w) \bigr) \right] ( \omega'' -1) +
 \omega' \opm_M^{\tilde{\gamma}-1} \bigl( g(w) \bigr) ( \omega'' -1) ,
\]
where the last term maps ${\cal H}^{s-\nu,\tilde{\gamma}}\bigl( (S^2)^\wedge \bigr)$ into
${\cal K}^{\infty,\tilde{\gamma}}\bigl( (S^2)^\wedge \bigr)$ and therefore does not contribute
to the asymptotics for sufficiently large $\tilde{\gamma}$. The first term in square brackets can be
computed like in the case of the $G_I$ Green operator, i.e.,
\begin{eqnarray}
\nonumber
 \lefteqn{ \omega' \left[ \opm_M^{\gamma-1} \bigl( g(w) \bigr) -
 \opm_M^{\tilde{\gamma}-1} \bigl( g(w) \bigr) \right] ( \omega'' -1) (u)} \\ \label{GII}
 & & \hspace{4cm} = \frac{1}{2\pi i} \omega' \int_{\Gamma_{\frac{3}{2}-\gamma} \cup \Gamma^{-1}_{\frac{3}{2}-\tilde{\gamma}}}
 r^{-w} g(w) \bigl( M_{\gamma-1} ( \omega'' -1) u \bigr)(w) dw .
\end{eqnarray}

\paragraph{Proof of Theorem \ref{theorem1}}
\begin{proof}
The explicit expressions for the asymptotic behaviour of the Green
operator ${ G}$ can be easily calculated from (\ref{hm1i0}),
(\ref{hm1i1}), (\ref{hm1i2}) using Eqs.~(\ref{GI}) and (\ref{GII}).
The details of the calculation are given in Appendix
\ref{AppendixA}. Collecting individual terms order by order with
respect to powers of $r$ and the angular momentum projection
operators $P_l$, $l=0,1,2\ldots$, we obtain the asymptotic formula
(\ref{Ghydrogen}).
\end{proof}

\appendix
\vspace{1cm}
\noindent {\Large {\bf Appendices}}
\section{Calculation of Green operators}
\label{AppendixA}
In this appendix, we characterize $G_I$ and $G_{II}$-type Green operators.
\subsection{$G_I$-type Green operator}
We begin with the $G_I$ operator
\begin{eqnarray*}
 G_I & = & Z \sum_{i=0}^N r^{i+1} \left[ \omega' \opm_M^{\gamma-2} \bigl( T^1 h^{(-1)}_i \bigr) \omega''
 - \omega' \opm_M^{\gamma-1} \bigl( T^1 h^{(-1)}_i \bigr) \omega'' \right] \\
 & & +E \sum_{i=0}^N r^{i+2} \left[ \omega' \opm_M^{\gamma-3} \bigl( h^{(-1)}_i \bigr) \omega''
 - \omega' \opm_M^{\gamma-1} \bigl( h^{(-1)}_i \bigr) \omega'' \right] ,
\end{eqnarray*}
where we have computed contributions up to $N=2$.

Using (\ref{hm1i0}) and (\ref{GI}), we obtain for $i=0$ with $v := \omega'' u$
\begin{eqnarray*}
 \lefteqn{Z r \left[ \omega' \opm_M^{\gamma-2} \bigl( T^1 h^{(-1)}_0 \bigr)
 - \omega' \opm_M^{\gamma-1} \bigl( T^1 h^{(-1)}_0 \bigr) \right](v)} \\
 & & \hspace{3cm} = Zr \omega' \frac{1}{2\pi i} \int_{\Gamma^{-1}_{\frac{3}{2}-\gamma} \cup \Gamma_{\frac{5}{2}-\gamma}}
 r^{-w} T^1 h^{(-1)}_0(w) \bigl( M_{\gamma -2} v \bigr)(w) dw \\
 & & \hspace{3cm} = -2Z \omega' P_0 \bigl( M v \bigr)(1) ,
\end{eqnarray*}
and
\begin{eqnarray*}
 \lefteqn{E r^2 \left[ \omega' \opm_M^{\gamma-3} \bigl( h^{(-1)}_0 \bigr)
 - \omega' \opm_M^{\gamma-1} \bigl( h^{(-1)}_0 \bigr) \right](v)} \\
 & & \hspace{3cm} = Er^2 \omega' \frac{1}{2\pi i} \int_{\Gamma^{-1}_{\frac{3}{2}-\gamma} \cup \Gamma_{\frac{7}{2}-\gamma}}
 r^{-w} h^{(-1)}_0(w) \bigl( M_{\gamma -3} v \bigr)(w) dw \\
 & & \hspace{3cm} = -E\omega' \left[ 2 P_0 \bigl( M v \bigr)(2)
 + \tfrac{2}{3} r  P_1 \bigl( M v \bigr)(1) \right] .
\end{eqnarray*}

For $i=1$, we get with (\ref{hm1i1})
\begin{eqnarray*}
 \lefteqn{Z r^2 \left[ \omega' \opm_M^{\gamma-2} \bigl( T^1 h^{(-1)}_1 \bigr)
 - \omega' \opm_M^{\gamma-1} \bigl( T^1 h^{(-1)}_1 \bigr) \right](v)} \\
 & & \hspace{3cm} = Zr^2 \omega' \frac{1}{2\pi i} \int_{\Gamma^{-1}_{\frac{3}{2}-\gamma} \cup \Gamma_{\frac{5}{2}-\gamma}}
 r^{-w} T^1 h^{(-1)}_1(w) \bigl( M_{\gamma -2} v \bigr)(w) dw \\
 & & \hspace{3cm} = -Z^2 \omega' \left[ -2 r P_0 \bigl( M v \bigr)(1)
 + \tfrac{2}{3} r P_1 \bigl( M v \bigr)(1) \right] ,
\end{eqnarray*}
and
\begin{eqnarray*}
 \lefteqn{E r^3 \left[ \omega' \opm_M^{\gamma-3} \bigl( h^{(-1)}_1 \bigr)
 - \omega' \opm_M^{\gamma-1} \bigl( h^{(-1)}_1 \bigr) \right](v)} \\
 & & \hspace{3cm} = Er^3 \omega' \frac{1}{2\pi i} \int_{\Gamma^{-1}_{\frac{3}{2}-\gamma} \cup \Gamma_{\frac{7}{2}-\gamma}}
 r^{-w} h^{(-1)}_2(w) \bigl( M_{\gamma -3} v \bigr)(w) dw \\
 & & \hspace{3cm} = -EZ \omega' \left[ -2 r P_0 \bigl( M v \bigr)(2)
 + \tfrac{2}{3} r P_1 \bigl( M v \bigr)(2)
 - \tfrac{1}{3} r^2 P_1 \bigl( M v \bigr)(1)
 + \tfrac{1}{5} r^2 P_2 \bigl( M v \bigr)(1) \right] .
\end{eqnarray*}

Finally, for $i=2$ we get with (\ref{hm1i2})
\begin{eqnarray*}
 \lefteqn{Z r^3 \left[ \omega' \opm_M^{\gamma-2} \bigl( T^1 h^{(-1)}_2 \bigr)
 - \omega' \opm_M^{\gamma-1} \bigl( T^1 h^{(-1)}_2 \bigr) \right](v)} \\
 & & \hspace{3cm} = Zr^3 \omega' \frac{1}{2\pi i} \int_{\Gamma^{-1}_{\frac{3}{2}-\gamma} \cup \Gamma_{\frac{5}{2}-\gamma}}
 r^{-w} T^1 h^{(-1)}_2(w) \bigl( M_{\gamma -2} v \bigr)(w) dw \\
 & & \hspace{3cm} = -Z \omega' \biggl[ \tfrac{2}{3} r^2 \bigl( Z^2 -E \bigr) P_0 \bigl( M v \bigr)(1)
 - \tfrac{1}{3} r^2 Z^2 P_1 \bigl( M v \bigr)(1)
 + \tfrac{1}{15} r^2 \bigl( Z^2 +2E \bigr) P_2 \bigl( M v \bigr)(1) \biggr] ,
\end{eqnarray*}
and
\begin{eqnarray*}
 \lefteqn{E r^4 \left[ \omega' \opm_M^{\gamma-3} \bigl( h^{(-1)}_2 \bigr)
 - \omega' \opm_M^{\gamma-1} \bigl( h^{(-1)}_2 \bigr) \right](v)} \\
 & & \hspace{3cm} = Er^4 \omega' \frac{1}{2\pi i} \int_{\Gamma^{-1}_{\frac{3}{2}-\gamma} \cup \Gamma_{\frac{7}{2}-\gamma}}
 r^{-w} h^{(-1)}_2(w) \bigl( M_{\gamma -3} v \bigr)(w) dw \\
 & & \hspace{3cm} = -E \omega' \biggl[ \tfrac{2}{3} r^2 \bigl( Z^2 -E \bigr) P_0 \bigl( M v \bigr)(2)
 - \tfrac{1}{3} r^2 Z^2 P_1 \bigl( M v \bigr)(2)
 + \tfrac{1}{15} r^2 \bigl( Z^2 +2E \bigr) P_2 \bigl( M v \bigr)(2) \biggr. \\
%& & \\
 & & \hspace{3cm} \,\,\,\,\,\, \biggl. + \tfrac{1}{15} r^3 \bigl( Z^2 -2E \bigr) P_1 \bigl( M v \bigr)(1)
 - \tfrac{1}{15} r^3 Z^2 P_2 \bigl( M v \bigr)(1)
 + \tfrac{2}{35} r^3 \bigl( \tfrac{1}{3}Z^2 +E \bigr) P_3 \bigl( M v \bigr)(1) \biggr] .
\end{eqnarray*}

\subsection{$G_{II}$-type Green operator}
According to our discussion in Section \ref{parametrix}, the $G_{II}$-type Green operator is given by
\begin{eqnarray*}
 G_{II} & = & \sum_{i=0}^N r^i \omega' \left[ \opm_M^{\gamma-1} \bigl( T^2 h^{(-1)}_i \bigr) -
 \opm_M^{\tilde{\gamma}-1} \bigl( T^2 h^{(-1)}_i \bigr)  \right] \bigl( \omega''
 -1 \bigr) \opm_M^{\gamma-1} \bigl( h_0 \bigr) \tilde{\omega} \\
 & & + Z \sum_{i=0}^N r^{i+1} \omega' \left[ \opm_M^{\gamma-1} \bigl( T^1 h^{(-1)}_i \bigr) -
 \opm_M^{\tilde{\gamma}-1} \bigl( T^1 h^{(-1)}_i \bigr) \right] \bigl( \omega''
 -1 \bigr) \tilde{\omega} \\
 & & + E \sum_{i=0}^N r^{i+2} \omega' \left[ \opm_M^{\gamma-1} \bigl( h^{(-1)}_i \bigr) -
 \opm_M^{\tilde{\gamma}-1} \bigl( h^{(-1)}_i \bigr) \right] \bigl( \omega''
 -1 \bigr) \tilde{\omega} \cdots ,
\end{eqnarray*}
where the dots indicate a remainder which does not contribute to the asymptotics. The parameter $\tilde{\gamma} > \gamma$ is
chosen sufficiently large such that all poles belong to the intervall
$[\frac{3}{2}-\tilde{\gamma}, \frac{3}{2}-\gamma]$.

Using (\ref{hm1i0}) and (\ref{GII}), we obtain for $i=0$ with $v := \tilde{\omega} u$
\begin{eqnarray*}
 \lefteqn{ \omega' \left[ \opm_M^{\gamma-1} \bigl( T^2 h^{(-1)}_0 \bigr) -
 \opm_M^{\tilde{\gamma}-1} \bigl( T^2 h^{(-1)}_0 \bigr) \right] ( \omega'' -1)
 \opm_M^{\gamma-1} \bigl( h_0 \bigr) (v)} \\
 & & \hspace{3cm} = \omega' \frac{1}{2\pi i} \int_{\Gamma_{\frac{3}{2}-\gamma} \cup \Gamma^{-1}_{\frac{3}{2}-\tilde{\gamma}}}
 r^{-w} T^2 h^{(-1)}_0(w) \left( M_{\gamma-1} ( \omega'' -1) \opm_M^{\gamma-1} \bigl( h_0 \bigr) (v) \right)(w) dw \\
 & & \hspace{3cm} = -\omega' \sum_{l=0}^L \frac{2}{2l+1} r^l P_l
 \left( M_{\gamma-1} ( \omega'' -1) \opm_M^{\gamma-1} \bigl( h_0 \bigr) (v) \right)(-l) ,
\end{eqnarray*}

\begin{eqnarray*}
 \lefteqn{ Zr \omega' \left[ \opm_M^{\gamma-1} \bigl( T^1 h^{(-1)}_0 \bigr) -
 \opm_M^{\tilde{\gamma}-1} \bigl( T^1 h^{(-1)}_0 \bigr) \right] ( \omega'' -1) (v)} \\
 & & \hspace{3cm} = Z r \omega' \frac{1}{2\pi i} \int_{\Gamma_{\frac{3}{2}-\gamma} \cup \Gamma^{-1}_{\frac{3}{2}-\tilde{\gamma}}}
 r^{-w} T^1 h^{(-1)}_0(w) \left( M_{\gamma-1} ( \omega'' -1) (v) \right)(w) dw \\
 & & \hspace{3cm} = -Z \omega' \sum_{l=1}^L \frac{2}{2l+1} r^l P_l
 \left( M_{\gamma-1} ( \omega'' -1) (v) \right)(-l+1) ,
\end{eqnarray*}

\begin{eqnarray*}
 \lefteqn{ E r^2 \omega' \left[ \opm_M^{\gamma-1} \bigl( h^{(-1)}_0 \bigr) -
 \opm_M^{\tilde{\gamma}-1} \bigl( h^{(-1)}_0 \bigr) \right] ( \omega'' -1) (v)} \\
 & & \hspace{3cm} = E r^2 \omega' \frac{1}{2\pi i} \int_{\Gamma_{\frac{3}{2}-\gamma} \cup \Gamma^{-1}_{\frac{3}{2}-\tilde{\gamma}}}
 r^{-w} h^{(-1)}_0(w) \left( M_{\gamma-1} ( \omega'' -1) (v) \right)(w) dw \\
 & & \hspace{3cm} = -E \omega' \sum_{l=2}^L \frac{2}{2l+1} r^l P_l
 \left( M_{\gamma-1} ( \omega'' -1) (v) \right)(-l+2) .
\end{eqnarray*}

For $i=1$, we get with (\ref{hm1i1})
\begin{eqnarray*}
 \lefteqn{ r \omega' \left[ \opm_M^{\gamma-1} \bigl( T^2 h^{(-1)}_1 \bigr) -
 \opm_M^{\tilde{\gamma}-1} \bigl( T^2 h^{(-1)}_1 \bigr) \right] ( \omega'' -1)
 \opm_M^{\gamma-1} \bigl( h_0 \bigr) (v)} \\
 & & \hspace{3cm} = r \omega' \frac{1}{2\pi i} \int_{\Gamma_{\frac{3}{2}-\gamma} \cup \Gamma^{-1}_{\frac{3}{2}-\tilde{\gamma}}}
 r^{-w} T^2 h^{(-1)}_1(w) \left( M_{\gamma-1} ( \omega'' -1) \opm_M^{\gamma-1} \bigl( h_0 \bigr) (v) \right)(w) dw \\
 & & \hspace{3cm} = Z \omega' \sum_{l=0}^L \frac{2}{(l+1)(2l+1)} r^{l+1} P_l
 \left( M_{\gamma-1} ( \omega'' -1) \opm_M^{\gamma-1} \bigl( h_0 \bigr) (v) \right)(-l) \\
 & & \hspace{3cm} \,\,\,\,\,\, -Z \omega' \sum_{l=1}^L \frac{2}{l(2l+1)} r^l P_l
 \left( M_{\gamma-1} ( \omega'' -1) \opm_M^{\gamma-1} \bigl( h_0 \bigr) (v) \right)(-l+1) ,
\end{eqnarray*}

\begin{eqnarray*}
 \lefteqn{ Zr^2 \omega' \left[ \opm_M^{\gamma-1} \bigl( T^1 h^{(-1)}_1 \bigr) -
 \opm_M^{\tilde{\gamma}-1} \bigl( T^1 h^{(-1)}_1 \bigr) \right] ( \omega'' -1) (v)} \\
 & & \hspace{3cm} = Z r^2 \omega' \frac{1}{2\pi i} \int_{\Gamma_{\frac{3}{2}-\gamma} \cup \Gamma^{-1}_{\frac{3}{2}-\tilde{\gamma}}}
 r^{-w} T^1 h^{(-1)}_1(w) \left( M_{\gamma-1} ( \omega'' -1) (v) \right)(w) dw \\
 & & \hspace{3cm} = Z^2 \omega' \sum_{l=1}^L \frac{2}{(l+1)(2l+1)} r^{l+1} P_l
 \left( M_{\gamma-1} ( \omega'' -1) (v) \right)(-l+1) \\
 & & \hspace{3cm} \,\,\,\,\,\, -Z^2 \omega' \sum_{l=2}^L \frac{2}{l(2l+1)} r^l P_l
 \left( M_{\gamma-1} ( \omega'' -1) (v) \right)(-l+2) ,
\end{eqnarray*}

\begin{eqnarray*}
 \lefteqn{ Er^3 \omega' \left[ \opm_M^{\gamma-1} \bigl( h^{(-1)}_1 \bigr) -
 \opm_M^{\tilde{\gamma}-1} \bigl( h^{(-1)}_1 \bigr) \right] ( \omega'' -1) (v)} \\
 & & \hspace{3cm} = E r^3 \omega' \frac{1}{2\pi i} \int_{\Gamma_{\frac{3}{2}-\gamma} \cup \Gamma^{-1}_{\frac{3}{2}-\tilde{\gamma}}}
 r^{-w} h^{(-1)}_1(w) \left( M_{\gamma-1} ( \omega'' -1) (v) \right)(w) dw \\
 & & \hspace{3cm} = EZ \omega' \sum_{l=2}^L \frac{2}{(l+1)(2l+1)} r^{l+1} P_l
 \left( M_{\gamma-1} ( \omega'' -1) (v) \right)(-l+2) \\
 & & \hspace{3cm} \,\,\,\,\,\, - EZ \omega' \sum_{l=3}^L \frac{2}{l(2l+1)} r^l P_l
 \left( M_{\gamma-1} ( \omega'' -1) (v) \right)(-l+3) .
\end{eqnarray*}

Finally for $i=2$, we get with (\ref{hm1i2})
\begin{eqnarray*}
 \lefteqn{ r^2 \omega' \left[ \opm_M^{\gamma-1} \bigl( T^2 h^{(-1)}_2 \bigr) -
 \opm_M^{\tilde{\gamma}-1} \bigl( T^2 h^{(-1)}_2 \bigr) \right] ( \omega'' -1)
 \opm_M^{\gamma-1} \bigl( h_0 \bigr) (v)} \\
 & & \hspace{3cm} = r^2 \omega' \frac{1}{2\pi i} \int_{\Gamma_{\frac{3}{2}-\gamma} \cup \Gamma^{-1}_{\frac{3}{2}-\tilde{\gamma}}}
 r^{-w} T^2 h^{(-1)}_2(w) \left( M_{\gamma-1} ( \omega'' -1) \opm_M^{\gamma-1} \bigl( h_0 \bigr) (v) \right)(w) dw \\
 & & \hspace{3cm} = -\omega' \sum_{l=0}^L \left( \frac{Z^2}{l+1} -E \right) \frac{2}{(2l+1)(2l+3)} r^{l+2} P_l
 \left( M_{\gamma-1} ( \omega'' -1) \opm_M^{\gamma-1} \bigl( h_0 \bigr) (v) \right)(-l) \\
 & & \hspace{3cm} \,\,\,\,\,\, + Z^2 \omega' \sum_{l=1}^L \frac{2}{l(l+1)(2l+1)} r^{l+1} P_l
 \left( M_{\gamma-1} ( \omega'' -1) \opm_M^{\gamma-1} \bigl( h_0 \bigr) (v) \right)(-l+1) \\
 & & \hspace{3cm} \,\,\,\,\,\, - \omega' \sum_{l=2}^L \left( \frac{Z^2}{l} +E \right) \frac{2}{(2l-1)(2l+1)} r^l P_l
 \left( M_{\gamma-1} ( \omega'' -1) \opm_M^{\gamma-1} \bigl( h_0 \bigr) (v) \right)(-l+2) ,
\end{eqnarray*}

\begin{eqnarray*}
 \lefteqn{ Zr^3 \omega' \left[ \opm_M^{\gamma-1} \bigl( T^1 h^{(-1)}_2 \bigr) -
 \opm_M^{\tilde{\gamma}-1} \bigl( T^1 h^{(-1)}_2 \bigr) \right] ( \omega'' -1) (v)} \\
 & & \hspace{3cm} = Z r^3 \omega' \frac{1}{2\pi i} \int_{\Gamma_{\frac{3}{2}-\gamma} \cup \Gamma^{-1}_{\frac{3}{2}-\tilde{\gamma}}}
 r^{-w} T^1 h^{(-1)}_2(w) \left( M_{\gamma-1} ( \omega'' -1) (v) \right)(w) dw \\
 & & \hspace{3cm} = -Z \omega' \sum_{l=1}^L \left( \frac{Z^2}{l+1} -E \right) \frac{2}{(2l+1)(2l+3)} r^{l+2} P_l
 \left( M_{\gamma-1} ( \omega'' -1) (v) \right)(-l+1) \\
 & & \hspace{3cm} \,\,\,\,\,\, + Z^3 \omega' \sum_{l=2}^L \frac{2}{l(l+1)(2l+1)} r^{l+1} P_l
 \left( M_{\gamma-1} ( \omega'' -1) (v) \right)(-l+2) \\
 & & \hspace{3cm} \,\,\,\,\,\, - Z \omega' \sum_{l=3}^L \left( \frac{Z^2}{l} +E \right) \frac{2}{(2l-1)(2l+1)} r^l P_l
 \left( M_{\gamma-1} ( \omega'' -1) (v) \right)(-l+3) ,
\end{eqnarray*}

\begin{eqnarray*}
 \lefteqn{ Er^4 \omega' \left[ \opm_M^{\gamma-1} \bigl( h^{(-1)}_2 \bigr) -
 \opm_M^{\tilde{\gamma}-1} \bigl( h^{(-1)}_2 \bigr) \right] ( \omega'' -1) (v)} \\
 & & \hspace{3cm} = E r^4 \omega' \frac{1}{2\pi i} \int_{\Gamma_{\frac{3}{2}-\gamma} \cup \Gamma^{-1}_{\frac{3}{2}-\tilde{\gamma}}}
 r^{-w} h^{(-1)}_2(w) \left( M_{\gamma-1} ( \omega'' -1) (v) \right)(w) dw \\
 & & \hspace{3cm} = -E \omega' \sum_{l=2}^L \left( \frac{Z^2}{l+1} -E \right) \frac{2}{(2l+1)(2l+3)} r^{l+2} P_l
 \left( M_{\gamma-1} ( \omega'' -1) (v) \right)(-l+2) \\
 & & \hspace{3cm} \,\,\,\,\,\, +EZ^2 \omega' \sum_{l=3}^L \frac{2}{l(l+1)(2l+1)} r^{l+1} P_l
 \left( M_{\gamma-1} ( \omega'' -1) (v) \right)(-l+3) \\
 & & \hspace{3cm} \,\,\,\,\,\, -E \omega' \sum_{l=4}^L \left( \frac{Z^2}{l} +E \right) \frac{2}{(2l-1)(2l+1)} r^l P_l
 \left( M_{\gamma-1} ( \omega'' -1) (v) \right)(-l+4) .
\end{eqnarray*}

\section{Tools for operators on a cone}
\label{AppendixB}
We complete here some details on the cone algebra that belong to the technical background of this paper.

\subsection{Mellin pseudo-differential operators on a cone}
\label{Mellin}
Before we start our discussion of Mellin pseudo-differential operators, let
us briefly recall the definition of ${\cal K}^{s,\gamma}(X^\wedge)$-spaces
for the case of $X=S^n$, the unit sphere in $\mathbb{R}^{1+n}$.
By the identification between $\mathbb{R}^{1+n} \setminus \{0\}$ and $X^\wedge$
via polar coordinates $\tilde{x} \rightarrow (r,x)$ we get
\[
 {\cal K}^{s,\gamma}(X^\wedge) = \omega {\cal H}^{s,\gamma}(X^\wedge) +(1-\omega) H^s(\mathbb{R}^{1+n}) .
\]
Here ${\cal H}^{s,\gamma}(X^\wedge) = r^\gamma {\cal H}^{s,0}(X^\wedge)$, and ${\cal H}^{s,0}(X^\wedge)$
for $s \in \mathbb{N}_0$ is defined to be the set of all $u(r,x) \in r^{-n/2} L^2(\mathbb{R}_+ \times X)$
such that $(r \partial_r)^jDu \in r^{-n/2} L^2(\mathbb{R}_+ \times X)$ for all $D \in \mbox{Diff}^{s-j}(X)$,
$0 \leq j \leq s$. The definition for $s \in \mathbb{R}$ in general follows by duality and complex interpolation.

In the definition of pseudo-differential operators on manifolds with conical singularity it is natural
to use the Mellin transform in the direction of the cone variable $r$.
Let us set $\Gamma_\beta := \{ w \in \mathbb{C} : \Re w = \beta \}$ for any real $\beta$.
We systematically employ the Mellin transform
\[
 \bigl( Mu \bigr) (w) = \int_0^\infty r^w u(r) \frac{dr}{r} ,
\]
which is for $u \in C^\infty_0(\mathbb{R}_+)$ an entire function in $w \in \mathbb{C},$ where
$Mu|_{\Gamma_\beta} \in {\cal S}(\Gamma_\beta)$ for every $\beta \in \mathbb{R}$, uniformly in
compact $\beta$-intervals. $M_\gamma u := Mu|_{\Gamma_{\frac{1}{2}-\gamma}}$ is called the weighted Mellin transform
of weight $\gamma \in \mathbb{R}$. We also admit functions/distributions in $r \in \mathbb{R}_+$ with values
in a Fr\'echet vector space, e.g., $C^\infty_0 \bigl( \mathbb{R}_+, C^\infty(X) \bigr)$. Then
$M_\gamma$, first applied on elements with compact support in $r \in \mathbb{R}_+$ extends to more general spaces.
For instance, when we define $\hat{H}^s(\Gamma_\beta \times X)$ to be the Fourier transform along $\Gamma_\beta$
of the space $H^s(\Gamma_\beta \times X)$ (the usual cylindrical Sobolev space of smoothness $s \in \mathbb{R}$)
then
\[
 M_{\gamma-\frac{n}{2}} : \omega {\cal K}^{s,\gamma}(X^\wedge) \rightarrow
 \hat{H}^s(\Gamma_{\frac{n+1}{2}-\gamma} \times X)
\]
is continuous for every $s \in \mathbb{R}$.

After these preparatory remarks, we can now introduce the concept of Mellin pseudo-differential operators on $X^\wedge$.
Let $L^\mu_{cl}(X, \mathbb{R}^l)$ denote the space of classical parameter-dependent pseudo-differential
operators of order $\mu \in \mathbb{R}$ on $X$, with parameter $\lambda \in \mathbb{R}^l$.
This is a Fr\'echet space in a natural way. When $\mathbb{R}$ is identified with $\Gamma_\beta$ for
some $\beta$ we also write $L^\mu_{cl}(X, \Gamma_\beta)$. For any operator-valued amplitude function
$f(w) \in L^\mu_{cl}(X, \Gamma_{\frac{n+1}{2} -\gamma})$, we obtain a Mellin pseudo-differential operator
\[
 \opm_M^{\gamma - \frac{n}{2}}(f) \omega u := M^{-1}_{\gamma - \frac{n}{2}}
 \bigl( f(w) M_{\gamma - \frac{n}{2}} \omega u(w) \bigr) .
\]
Multiplied by another cut-off function $\tilde{\omega}$ from the left this gives us a continuous
operator
\[
 \tilde{\omega} \opm_M^{\gamma - \frac{n}{2}}(f) \omega :
 {\cal K}^{s,\gamma}(X^\wedge) \rightarrow {\cal K}^{s-\mu,\gamma}(X^\wedge)
\]
for every $s \in \mathbb{R}$. It is known that such a continuity also holds for
$f=f(r,w) \in C^\infty \bigl( \overline{\mathbb{R}}_+, L^\mu_{cl}(X, \Gamma_{\frac{n+1}{2} -\gamma}) \bigr)$.
An example is
\[
 f(r,w) = \sum_{j=0}^\mu a_j(r) w^j|_{\Gamma_{\frac{n+1}{2} -\gamma}} ,
\]
compare with the formula (\ref{Amu}). The corresponding polynomial in $w \in \mathbb{C}$ gives rise to a
holomorphic function with values in $\mbox{Diff}^\mu(X)$, smoothly depending on $r \in \overline{\mathbb{R}}_+$.
Let $M_{\cal O}^\mu (X)$ for any $\mu \in \mathbb{R}$ denote the subspace of all
$h \in {\cal A} \bigl( \mathbb{C}, L^\mu_{cl}(X) \bigr)$ such that $h|_{\Gamma_\beta} \in  L^\mu_{cl}(X, \Gamma_\beta)$
for every $\beta \in \mathbb{R}$, uniformly in compact $\beta$-intervals. The space
$M_{\cal O}^\mu (X)$ is Fr\'echet in a natural way, and it makes sense to talk about
$C^\infty \bigl( \overline{\mathbb{R}}_+, M_{\cal O}^\mu (X) \bigr)$. In particular, we have
$\sum_{j=0}^\mu a_j(r) w^j \in M_{\cal O}^\mu (X)$.

Let us denote by $\sigma_\psi(A)$ the homogeneous principal symbol of $A$ of oder $\mu$ as an invariantly defined
function on $T^\ast X^\wedge \setminus 0$. Then in the variables $(r,x)$ and covariables $(\rho,\xi)$, we
also have, what we call the reduced principal symbol,
\[
 \tilde{\sigma}_\psi(A)(r,x,\rho,\xi) = r^\mu \sigma_\psi(A)(r,x,r^{-1}\rho,\xi) .
\]
By ellipticity of $A$ with respect to $\sigma_\psi$ (briefly, $\sigma_\psi$-ellipticity), we understand
the condition $\sigma_\psi(A) \neq 0$ on $T^\ast X^\wedge \setminus 0$ and
$\tilde{\sigma}_\psi(A)(r,x,\rho,\xi) \neq 0$ for $(\rho,\xi) \neq 0$ up to $r=0$.
In this case, we always know that the principal conormal symbol
\begin{equation}
 \sigma_M(A)(w) := \sum_{j=0}^\mu a_j(0) w^j : H^s(X) \rightarrow H^{s-\mu}(X)
\label{sigmac}
\end{equation}
is a family of Fredholm operators, and that there is a sequence $(p_j)_{j \in \mathbb{Z}} \subset \mathbb{C}$
such that $|\Re p_j| \rightarrow \infty$ as $|j| \rightarrow \infty$, and (\ref{sigmac}) bijective for all
$w \in \mathbb{C} \setminus (p_j)_{j \in \mathbb{Z}}$. Now $A$ is said to be elliptic with respect
to a weight $\gamma \in \mathbb{R}$ if $(p_j)_{j \in \mathbb{Z}} \cap \Gamma_{\frac{n+1}{2} -\gamma} = \emptyset$.
Then (\ref{sigmac}) is a family of isomorphisms for all $w \in \Gamma_{\frac{n+1}{2} -\gamma}$.
Furthermore, if in addition $A$ is exit-elliptic (cf.~\cite{Schulze98}; some details are also outlined in \cite{FSS08})
which is a condition for $r \rightarrow \infty$ then
\begin{equation}
 A: {\cal K}^{s,\gamma}(X^\wedge) \rightarrow {\cal K}^{s-\mu,\gamma -\mu}(X^\wedge)
\label{Aelliptic}
\end{equation}
is a Fredholm operator for all $s \in \mathbb{R}$. Actually, it can be proved that the ellipticity conditions
are even necessary for the Fredholm property of (\ref{Aelliptic}).

It is a basic task (not only for the specific
application in the present paper) to assess asymptotics of solutions to an equation $Au=f$ for $r \rightarrow 0$
when $f \in {\cal K}^{s-\mu,\gamma -\mu}(X^\wedge)$ has an asymptotic expansion of the kind
\begin{equation}
 f(r,x) \sim \sum_j \sum_{k=0}^{m_j} d_{jk}(x) r^{-p_j} \log^kr
\label{fasymp}
\end{equation}
for $r \rightarrow 0$, with coefficients $d_{jk} \in C^\infty(X)$.
In the case of our application we have $f=0$ which is, of course, a special case.
It turns out that any solution $u(r,x)$ has an analogous asymptotic expansion as (\ref{fasymp}),
where the points $p \in (p_j)_{j \in \mathbb{Z}}$ from the system above with $\Re p < \frac{n+1}{2} -\gamma$
contribute further exponents to the resulting asymptotic expansion.
The generalities may be found in \cite{Schulze98}. Here we sketch a few details, mainly to explain notation that
played a role before.

\subsection{The structure of parametrices}
\label{structurepara}
Asymptotics of solutions may be regarded as an aspect of elliptic regularity, here in subspaces
${\cal K}^{s,\gamma}_P(X^\wedge) \subset {\cal K}^{s,\gamma}(X^\wedge)$ of functions $u(r,x)$ with
asymptotics of type $P$, i.e.,
\[
 u(r,x) \sim \sum_j \sum_{k=0}^{m_j} c_{jk}(x) r^{-p_j} \log^kr
\]
as $r \rightarrow 0$. The asymptotic type $P= \{(p_j,m_j)\}_{j \in \mathbb{N}_0} \subset \mathbb{C} \times \mathbb{N}_0$
is a sequence of pairs where $\Re p_j \rightarrow -\infty$ as $j \rightarrow \infty$, and $c_{jk} \in C^\infty(X)$.
In the notion of asymptotics, we may also include the coefficients and control the finite dimensional subspaces
of $C^\infty(X)$. For simplicity, we drop this aspect in the general description.

The weighted Mellin transform $M_{\gamma- \frac{n}{2}}, n=\mathrm{dim}\,X,$ transforms $\omega {\cal K}^{s,\gamma}_P(X^\wedge)$
for any cut-off function $\omega$ to a space of functions that are meromorphic in $w \in \mathbb{C}$ with
poles at $p_j$ of multiplicity $m_j+1$. This picture fits to Mellin operators with meromorphic symbols.
Given a sequence $Q := \{(q_j,n_j)\}_{j \in \mathbb{Z}} \subset \mathbb{C} \times \mathbb{N}_0$
with $|\Re q_j| \rightarrow \infty$ as $|j| \rightarrow \infty$, by $M^{-\infty}_Q(X)$ we denote a suitable
space of functions $f(w) \in {\cal A} \bigl( \mathbb{C} \setminus \pi_\mathbb{C} Q, L^{-\infty}(X) \bigr)$
for $\pi_\mathbb{C} Q := \{ q_j : j \in \mathbb{Z} \}$ which are meromorphic with poles at $q_j$ of multiplicity
$n_j+1$. We call $Q$ a Mellin asymptotic type.
Similarly as for functions we can also control the coefficients in the principal parts of the Laurent expansions
to be finite rank smoothing operators on $X$; however, again for simplicity, we ignore this aspect in our
discussion. The space $M^{-\infty}_Q(X)$ is defined by the condition that for every $f\in M^{-\infty}_Q(X)$ and  $\pi_\mathbb{C} Q$-excision
function $\chi$
\[
 \chi f |_{\Gamma_\beta} \in L^{-\infty}(X,\Gamma_\beta)
\]
for every $\beta \in \mathbb{R}$, uniformly in compact $\beta$-intervals. Moreover, we set
\begin{equation}
 M^\mu_Q(X) := M^\mu_{{\cal O}}(X) + M^{-\infty}_Q(X).
\label{MmuQ}
\end{equation}
A basic observation is that when $A$ is elliptic with respect to $\sigma_\psi$, there is a Mellin asymptotic
type $Q$ such that
\[
 \sigma_M(A)^{-1}(w) \in M^{-\mu}_Q(X) .
\]
This is just what we exploit in our paper. Another aspect of our method is to use the existence of a
parametrix $P$ of $A$ which is, roughly speaking, of the form
\[
 P = r^\mu \opm_M^{(\gamma - \mu) - \frac{n}{2}}(h) + \ \mbox{remainder}
\]
for a function $h(r,z) \in C^\infty \bigl( \overline{\mathbb{R}}_+, M^{-\mu}_Q(X) \bigr)$
with a  suitable Mellin asymptotic type $Q$. The remainder mainly consists of a so-called
Green operator. Such an operator is defined by its mapping properties, namely, to map
${\cal K}^{s-\mu,\gamma -\mu}(X^\wedge)$ to a space of the type ${\cal S}^\gamma_P(X^\wedge)$ for
a suitable asymptotic type $P$ (and a similar condition for the formal adjoint).
In the more ``orthodox'' cone algebra that may be found, e.g., in \cite{Schulze98}, one refers
to a decomposition of $h$ similarly as (\ref{MmuQ}), according to
\[
 C^\infty \bigl( \overline{\mathbb{R}}_+, M^{-\mu}_Q(X) \bigr) =
 C^\infty \bigl( \overline{\mathbb{R}}_+, M^{-\mu}_{{\cal O}}(X) \bigr) +
 C^\infty \bigl( \overline{\mathbb{R}}_+, M^{-\infty}_Q(X) \bigr) ,
\]
and localises the Mellin operators near $0$ by cut-off functions on both sides.
What we do here is to find the coefficients of the Taylor expansion of $h(r,w)$
for $r \rightarrow 0$ which belong to $M^{-\mu}_Q(X)$ and to compose operators
localized at $0$, cf.~(\ref{PA}), where we employ commutation properties between operators with
meromorphic Mellin symbols and powers of $r$, cf.~(\ref{gohar}), which leaves Green remainders. The latter ones are just
explicitly computed in our example.

\end{document}